\documentclass[a4paper,11pt,notitlepage]{amsart}
\usepackage{amsmath,amssymb,mathrsfs,amsthm}
\usepackage{verbatim}
\usepackage[spanish,USenglish]{babel} 
\usepackage{color}
\usepackage{tikz}
\usepackage{graphicx}

\newcommand{\blue}[1]{\textcolor{blue}{#1}}

\newtheorem{theorem}{Theorem}

\newtheorem{lemma}[theorem]{Lemma}

\newtheorem{proposition}[theorem]{Proposition}
\newtheorem{definition}[theorem]{Definition}

\newtheorem{remark}[theorem]{Remark}

\newcommand{\R}{\mathbb{{R}}}
\newcommand{\Z}{\mathbb{{Z}}}
\newcommand{\N}{\mathbb{{N}}}

\newcommand{\C}{\mathbb{{C}}}

\begin{document}

\title[Fractional first order differences involving diffusion semigroups]{Non-local fractional derivatives. Discrete and continuous}

\author[Abadias]{Luciano Abadias}
\address{Departamento de Matem\'aticas, Instituto Universitario de Matem\'aticas y Aplicaciones, Universidad de Zaragoza, 50009 Zaragoza, Spain.}
\email{labadias@unizar.es}

\author[De Le\'on]{Marta De Le\'on}
\address{Departamento de Matem\'aticas, Facultad de Ciencias, Universidad Aut\'onoma de Madrid, 28049 Madrid, Spain.}
\email{marta.leon@uam.es}

\author[Torrea]{Jos\'e L. Torrea}
\address{Departamento de Matem\'aticas, Facultad de Ciencias, Universidad Aut\'onoma de Madrid, 28049 Madrid, Spain, and Instituto de Ciencias Matem\'aticas (CSIC-UAM-UC3M-UCM).}
\email{joseluis.torrea@uam.es}

\thanks{L. Abadias has been partially supported  by Project MTM2013-42105-P, DGI-FEDER, of the MCYTS, and Project E-64, D.G. Arag\'on. Second and third authors have been partially supported by MTM2015-66157-C2-1-P}



\begin{abstract} We prove  maximum and comparison principles for fractional discrete derivatives in the integers. Regularity results when the space is a mesh of length $h$, and approximation theorems to the continuous fractional derivatives are shown. When the functions are good enough,  these approximation procedures give a measure of the order of approximation.  These results also allows us to prove the coincidence, for good enough functions, of the Marchaud and Gr\"unwald-Letnikov derivatives in every point and the speed of convergence to the Gr\"unwald-Letnikov derivative.  The fractional discrete derivative will be also described as a Neumann-Dirichlet operator defined by a semi-discrete extension problem. Some operators related to the Harmonic Analysis associated to the discrete derivative will be also considered, in particular their behavior in the Lebesgue spaces $\ell^p(\Z).$
\end{abstract}

\date{}

\maketitle

\section{Introduction}

Fractional derivatives on time have been used to propose nonlocal models to describe non-diffusive transport
in magnetically confined plasmas, see \cite{Castillo}. They also appear in the study of parabolic problems in which it is natural to take into account the past, see\cite{Caffarelli}. Some porous medium flow with fractional time derivative have been considered recently, see \cite{Caffarelli2}. In these cases, regarding the fractional derivative on time, several discretization techniques play a crucial role. In the literature,  we can find different representations of the classical definition of the Riemann-Liouville fractional derivative as Caputo, Marchaud or Gr\"unwald-Letnikov, see \cite{Samko}. In general they no coincide, for example they have different behavior respect to constant functions.

In this paper we  study  fractional discrete derivatives. We shall prove  maximum and comparison principles as well as regularity results. Also  approximation theorems to the continuous fractional derivatives will be shown. When the functions are good enough,  these approximation procedures give a measure of the order of approximation.  These results also allows us to prove the coincidence, for good enough functions,  of the Marchaud and Gr\"unwald-Letnikov derivatives in every point and the speed of convergence to the Gr\"unwald-Letnikov derivative.  The fractional discrete derivative will be also described as a Neumann-Dirichlet operator defined by a semi-discrete extension problem. Some operators related to the Harmonic Analysis associated to the discrete derivative will be also considered, in particular their behavior in the Lebesgue spaces $\ell^p(\Z).$

As far as we know   S. Chapman in 1911, see \cite{Chapman}, was the first author who consider ``differences of fractional order''. See also the paper of 1956 by B. Kuttner, \cite{Kuttner}. For $s>0$, given a sequence $a_n$ they define\begin{equation}\label{Chapman}\triangle^s a_n= \sum_{m=0}^{\infty} \binom{-s-1+m}{m} a_{n+m}.
\end{equation}
Their motivation was to extend the obvious formula for differences of natural order. Also, by using Fourier transform in the integers it can be  checked that the above definition produces $\widehat{(\triangle^s a_n)}(\theta) = (1-e^{i\theta})^s \hat{a_n}(\theta)$ in coherence with the fact $\widehat{(a_n-a_{n+1})}(\theta) = (1-e^{i\theta})\hat{a_n}(\theta).$

Observe that Chapman's definition only  cares about   the future, but we shall also consider the discrete derivatives which cares about   the past. For $f:\Z\rightarrow\R$ ,  we define
 {\sl ``the discrete derivative from the right''} and {\sl ``the discrete
derivative from the left''} as the operators given by the formulas
\begin{equation}
\delta_{\rm right}f(n)=f(n)-f(n+1)\:\: and\:\:\delta_{\rm left}f(n)=f(n)-f(n-1).
\end{equation}

Given the function  $G_t(n)=e^{-t}\frac{t^{n}}{n!},\, n\in\N_0$, we define
\begin{equation}\label{uu}
u(n,t) = \displaystyle\sum_{j=0}^\infty G_t(j)f(n+j).
\end{equation}
It can be seen, see Section \ref{Diffusion}, that $u$ satisfies the following semi-discrete transport equation
\begin{equation}\label{eq1}
\left\{\begin{array}{ll}
\partial_t u(n,t)+ \delta_{\rm right}u(n,t) =0,&n\in\Z,\,t\geq 0,\\
u(n,0)=f(n),&n\in\Z.
\end{array} \right.
\end{equation}
The function $v(n,t)= \displaystyle\sum_{j=0}^\infty G_t(j)f(n-j)$ satisfies the analogous equation for  $\delta_{\rm left}.$
These equations will drive us to define the following nonlocal fractional operators
\begin{align*}
(\delta_{\rm right})^\alpha f(n)=\frac{1}{\Gamma(-\alpha)}\int_0^\infty \frac{u(n,t) -f(n)}{t^{1+\alpha}}dt, \  0< \alpha < 1,
\end{align*}
and
\begin{align*}
(\delta_{\rm right})^{-\alpha} f(n)=\frac{1}{\Gamma(\alpha)}\int_0^\infty \frac{u(n,t)}{t^{1-\alpha}}dt, \  0< \alpha < 1,
\end{align*}
and the corresponding formula for $(\delta_{\rm left})^\alpha, \, -1< \alpha <1.$
We shall prove that this definition of $(\delta_{\rm right})^\alpha$ coincides with the definition (\ref{Chapman}) given by Chapman in 1911, see Section \ref{Diffusion}.

 Before  the concrete presentation of our results, we observe that the definitions above can be given for a mesh with step length $h>0$ instead of the integers mesh with step length $1$. In other words, we can work in the field $\mathbb{Z}_h= \{ jh: j\in \Z\}$. In this way we  define
$$
\delta_{\rm right}u(hn)=\frac{u(hn)-u(h(n+1))}{h},\quad \delta_{\rm left}u(hn)=\frac{u(hn)-u(h(n-1))}{h},\quad n\in\Z, $$
and
the associated   $G_{t/h}(j)$

Now we state the main results of this paper.
\begin{theorem}\label{'teo22'} Let $0<\alpha<1$ and $1\leq p\leq \infty.$ \begin{itemize}
\item[(i)] Let $u\in\ell^p(\Z_h)$ such that $u(j_0h)=0$ for some $j_0\in\Z,$ and $u(jh)\geq 0$ for all $j_0\leq j.$ Then $(\delta_{\rm right})^{\alpha}u(j_0h)\leq 0.$ Moreover, $(\delta_{\rm right})^{\alpha}u(j_0h)= 0$ if and only if $u(jh)=0$ for all $j_0\leq j.$

\item[(ii)] Let $u,v\in\ell^p(\Z_h)$ such that $u(j_0h)=v(j_0h)$ for some $j_0\in\Z,$ and $u(jh)\geq v(jh)$ for all $j_0\leq j.$ Then $(\delta_{\rm right})^{\alpha}u(j_0h)\leq (\delta_{\rm right})^{\alpha}v(j_0h).$ Moreover, $(\delta_{\rm right})^{\alpha}u(j_0h)=(\delta_{\rm right})^{\alpha}v(j_0h)$ if and only if $u(jh)=v(jh)$ for all $j_0\leq j.$
\end{itemize}
\end{theorem}

The following result is a consequence of the above theorem.

\begin{theorem}\label{'cor23'} Let $j_0<j_1\in\Z$ and $u,v\in\ell^p(\Z_h)$ with $1\leq p\leq \infty.$\begin{itemize}
\item[(i)] Let $u$ be a solution of \begin{equation*}
\left\{\begin{array}{ll}
(\delta_{\rm right})^{\alpha}u=f,& \text{in }[j_0,j_1)\\ \\
u=0,&\text{in }[j_1,\infty).
\end{array} \right.
\end{equation*}
If $f\geq 0$ in $[j_0,j_1)$ then $u\geq 0$ in $[j_0,\infty).$

\item[(ii)] If $(\delta_{\rm right})^{\alpha}u\leq 0$ in $[j_0,j_1)$ and $u\leq 0$ in $[j_1,\infty),$ then $$\sup_{j\geq j_0}u(jh)=\sup_{j\geq j_1}u(jh).$$

\item[(iii)] If $(\delta_{\rm right})^{\alpha}u\geq 0$ in $[j_0,j_1)$ and $u\geq 0$ in $[j_1,\infty),$ then $$\inf_{j\geq j_0}u(jh)=\inf_{j\geq j_1}u(jh).$$

\item[(iv)] If \begin{equation*}
\left\{\begin{array}{ll}
(\delta_{\rm right})^{\alpha}u\geq (\delta_{\rm right})^{\alpha}v,& \text{in }[j_0,j_1)\\ \\
u\geq v,&\text{in }[j_1,\infty),
\end{array} \right.
\end{equation*}
then $u\geq v$ in $[j_0,\infty).$ In particular, we have uniqueness of the Dirichlet problem \begin{equation*}
\left\{\begin{array}{ll}
(\delta_{\rm right})^{\alpha}u=f,& \text{in }[j_0,j_1)\\ \\
u=g,&\text{in }[j_1,\infty).
\end{array} \right.
\end{equation*}
\end{itemize}
\end{theorem}

Also we shall get regularity results for the discrete H\"older classes $C_h^{k,\beta}$, see Theorem \ref{'teo24'}. These results drive us to compare the discrete fractional derivatives with the discretization of the continuous fractional derivatives as defined in \cite{Caffarelli,Caffarelli2,Bernardis}.
Given a function $u$ defined on $\R,$ we consider its restriction $r_h u$ (or discretization) to $\Z_h,$ that is, $r_h u(j)=u(hj)$ for $j\in\Z.$ We have the following theorems.
\begin{theorem} \label{'teo26'}Let $0<\beta\leq 1$ and $0<\alpha<1.$ \begin{itemize}
\item[(i)] Let $u\in C^{0,\beta}(\R)$ and $\alpha<\beta.$ Then $$\lVert (\delta_{\rm right})^{\alpha}(r_h u)-r_h ((D_{\text{right}})^{\alpha}u)\rVert_{\ell^{\infty}}\leq C_{\alpha} [u]_{C^{0,\beta}(\R)}h^{\beta-\alpha}.$$
\item[(ii)] Let $u\in C^{1,\beta}(\R)$ and $\alpha<\beta$. Then $$\lVert -\delta_{\rm right}(\delta_{\rm right})^{\alpha}(r_h u)-r_h (\frac{d}{dx}(D_{\text{right}})^{\alpha}u)\rVert_{\ell^{\infty}}\leq C_{\alpha} [u]_{C^{1,\beta}(\R)}h^{\beta-\alpha}.$$
\end{itemize}
Here, the operators $(D_{\text{right/left}})^{\alpha}$ are the Marchaud derivatives, see \cite{Samko}, that is, \begin{equation}\label{Marchaud}(D_{\text{right/left}})^{\alpha}f(x)=\frac{1}{\Gamma(-\alpha)}\int_0^{\infty}\frac{f(x\pm t)-f(x)}{t^{1+\alpha}}dt.\end{equation}
The classes $C^{k,\beta}(\R),\, k\in \N_0,\, \beta >0,$ are the usual H\"older classes on the real line (see Section \ref{PDE}).
\end{theorem}
\noindent There are analogous results when we substitute $\delta_{\rm left}$ by $\delta_{\rm right}$ and $D_{\text{left}}$ by $D_{\text{right}}$ respectively.

 Consider the {\it fractional differences of order $\alpha$}, with $\alpha>0$
\begin{equation*}
\Delta_{h,\pm}^\alpha f(x)=\sum_{k=0}^\infty (-1)^k \binom{\alpha}{k}f(x\pm kh)=\sum_{k=0}^\infty\Lambda^\alpha(k)f(x\pm kh), \, x\in\R, h>0.
\end{equation*}
The  {\it Gr\"unwald-Letnikov derivatives} of a function $f$ are  defined by
\begin{equation}\label{GL}f^{\alpha}_\pm(x) =\lim_{h\rightarrow +0}\frac{\Delta_{h,\pm}^\alpha f(x)}{h^\alpha}, \end{equation}
see \cite[pages 371--373]{Samko}.

The coincidence of the Marchaud and the Gr\"unwald-Letnikov derivatives is known in almost everywhere sense or in $L^p(\R), 1\le p < \infty$,  for $f\in L^r(\R)$,  with $r$ and $p$ independent, see \cite[Theorems 20.2 ,20.4]{Samko}. As a consequence  of our Theorem \ref{'teo26'} we shall prove that, for good enough functions, both derivatives coincide pointwise. Moreover we get  the speed of convergence of the limit in (\ref{GL}), which is of order $h^{\beta-\alpha}.$

\begin{theorem}\label{GL-M} Let $0<\alpha < \beta \le 1$ and $f\in C^{0,\beta}(\R).$ Then $f^{\alpha}_\pm(x) = (D_{\text{right/left}})^{\alpha}f(x)$ for every point $x\in \mathbb{R}$. Moreover, there exists a positive constant $C_{\alpha,\beta}$ such that
$$\left|(D_{\text{right/left}})^{\alpha}f(x)-\frac{\Delta_{h,\pm}^{\alpha}f(x)}{h^\alpha}\right|\leq C_{\alpha,\beta} [f]_{C^{0,\beta}(\R)}h^{\beta-\alpha},\quad x\in\R.$$
\end{theorem}

The operators $(\delta_{\rm right/left})^\alpha$ are non-local, but they can be understood as limits (when $t\to0$) of a local extension problem on $\R_+\times\Z$. In fact, by following the ideas of \cite{CS,ST, Gale} we get the next result.

\begin{theorem}\label{gale} Let $f\in \ell^p(\Z)$ and $0<\gamma<1.$ Consider the equation \begin{equation}\label{extension}
\partial^2_{zz}U(z,\cdot) + \frac{1-2 \gamma}{z}\partial_zU(z,\cdot) - \delta_{\rm right}U(z,\cdot)=0, \quad z\in S_{\pi/4},\end{equation} where $S_{\pi/4}=\{z\in\C\,|\, z\neq 0 \text{ and }|\arg z|<\pi/4 \}.$ The formula \begin{equation}\label{solution}U(z,\cdot)\,=\,\frac{z^{2\gamma}}{4^{\gamma}\Gamma(\gamma)}\int_0^\infty e^{-\frac{z^2}{4t}}u(\cdot,t)\frac{dt}{t^{1+\gamma}}\,=\,\frac{1}{\Gamma(\gamma)}\int_0^\infty e^{-\frac{z^2}{4t}}u(\cdot,t) \frac{dt}{t^{1-\gamma}}\end{equation}
 solves \eqref{extension} on $\ell^p(\Z)$, where $u(\cdot,t)$ is defined in (\ref{uu}) .
 Moreover, $$\lim_{z\to 0}U(z,\cdot)=f(\cdot)\ \text{ and }\ \frac{1}{2\gamma}\lim_{z\to 0}z^{1-2\gamma}\partial_zU(z,\cdot)=\frac{\Gamma(-\gamma)}{4^{\gamma}\Gamma(\gamma)}(\delta_{\rm right})^\gamma f(\cdot),$$ where both limits hold through proper subsectors of $S_{\pi/4}$ in the $\ell^p(\Z)$ sense.
 \newline A parallel result can be stated for $\delta_{\rm left}.$

\end{theorem}


%
\begin{remark} {\bf Extension problem for negative powers.}
It is clear that if a function $g$ is good enough, and in \eqref{solution} we substitute $f$ by $(\delta_{\rm right})^{-\gamma}g$ and $(\delta_{\rm right})^\gamma f$ by $g,$  $U$ solves the same equation with the initial Dirichlet condition  $(\delta_{\rm right})^{-\gamma}g$ and the Neumann condition $\frac{\Gamma(-\gamma)}{4^{\gamma}\Gamma(\gamma)}g.$ See \cite{Caffarelli3,Gale,ST}. \end{remark}

\begin{remark}
We will also see that the formula \eqref{solution} gives an explicit solution of \eqref{extension} in the classical sense for appropriate functions $f$ and $g,$ see  Section \ref{Sextension}.  In fact the formula provides an expression  of the  Poisson semigroup associated to $\delta_{\rm right}$, see Section \ref{secciondura}.
\end{remark}
\noindent Several versions of this theorem have been studied in \cite{Bernardis}.

The paper also contains results in the field of Harmonic Analysis related to the operators $\delta_{\rm right/left}$. In particular we analyze   maximal operators and Littlewood-Paley square functions associated to the heat and Poisson semigroups naturally linked to $\delta_{\rm right/left}$. In the case of the heat semigroup both, the maximal function and the square function, have a bad behavior on $\ell^p(\Z),$ see Section \ref{secciondura}. However the maximal operator and the square function associated to the Poisson  semigroups have suitable boundedness properties in $\ell^p(\Z)$.

We have the following result.
\begin{theorem}\label{'teo20'}
 Let $1< p<\infty.$ Let $S$ be either the maximal function or the square Littlewood-Paley function associated to the Poisson semigroups defined in (\ref{P.K.}). Then $S$ is bounded from $\ell^p(\Z)$ into itself and from $\ell^1(\Z)$ into weak-$\ell^1(\Z)$ (For the maximal function $p$ can be $\infty$).
\end{theorem}

For the proof of this Theorem we shall use vector-valued Calder\'on-Zygmund theory. In short, we could say that the Calder\'on-Zygmund Theory relays in the proof of two facts.  First the boundedness of the operator in a concrete Lebesgue space $\ell^p(\Z)$ for some value of $p, 1<p\le \infty$. Secondly an estimate about the smoothness of the kernel. In our case both items are nontrivial. The key is the appropriate use of boundedness of the MacDonald function, with complex and real arguments, involving Cauchy integration. We think that these proofs will be  of independent interest for the reader.

Throughout the paper, we use the variable constant convention, in which $C$ denotes a constant which may not be the same from line to line. The constant is frequently written with subindexes to emphasize that it depends on some parameters.

\section{The discrete fractional derivatives via a semigroup language} \label{Diffusion}
\setcounter{theorem}{0}
\setcounter{equation}{0}

We shall use semigroup language as an alternative approach to these differences of fractional order.
Given the function  $G_t(n)=e^{-t}\frac{t^{n}}{n!},\, n\in\N_0$, we define the operators
\begin{equation}\label{semigrupos}T_{t,+}f(n)=\displaystyle\sum_{j=0}^\infty G_t(j)f(n+j), \hbox{ and } T_{t,-}f(n)=\displaystyle\sum_{j=0}^\infty G_t(j)f(n-j),\qquad t>0,\,n\in\Z.\end{equation}
In this section we mainly prove that $T_{t,\pm}f(n)$ are markovian semigroups on $\ell^p(\mathbb{Z}),    \, 1\leq p\leq \infty$,  whose infinitesimal generators are $-\delta_{\rm right}$ and $-\delta_{\rm left}$ and that the function $u(n,t) = T_{t,+}f(n)$ satisfies the  Cauchy problem \eqref{eq1}.


\begin{lemma}\label{'lema27'}
Let $f\in\ell^\infty(\Z)$ and  $\{T_{t,\pm}\}_{t\geq 0}$ be the families defined in (\ref{semigrupos}), then
$$
\displaystyle\lim_{t\to 0}\frac{T_{t,+}f(n)-f(n)}{t}=-\delta_{\rm right}f(n)\:\:and\:\;\displaystyle\lim_{t\to 0}\frac{T_{t,-}f(n)-f(n)}{t}=-\delta_{\rm left}f(n),\:\;n\in\Z.
$$
\end{lemma}
\begin{proof}
Let $f\in\ell^\infty(\mathbb{Z})$ and $n\in\Z$. Observe that
\begin{align*}
\frac{T_{t,+}f(n)-f(n)}{t}&=\frac{1}{t}e^{-t}\sum_{j=0}^\infty \frac{t^j}{j!}\left( f(n+j)-f(n)\right)=e^{-t}\sum_{j=1}^\infty \frac{t^{j-1}}{j!}\left( f(n+j)-f(n)\right)\\
&=e^{-t}(f(n+1)-f(n))+e^{-t}\sum_{j=2}^\infty \frac{t^{j-1}}{j!}\left( f(n+j)-f(n)\right).
\end{align*}
 Dominated Convergence gives the first identity in the statement.
By a similar argument,
$$
\displaystyle\lim_{t\to 0}\frac{T_{t,-}f(n)-f(n)}{t}=\lim_{t\to 0}e^{-t}(f(n-1)-f(n))+\lim_{t\to 0}e^{-t}\sum_{j=2}^\infty \frac{t^{j-1}}{j!}\left( f(n-j)-f(n)\right)=-\delta_{\rm left}f(n).
$$
\end{proof}
The next Proposition shows that although the semigroups are not self adjoint, they satisfy the rest of the properties of the so called symmetric diffusion semigroups in the sense of E. M. Stein, see \cite{Stein}.
\begin{proposition}\label{'prop18'} Let $f\in \ell^{p}(\Z)$ with $1\leq p\leq \infty.$ The families of operators $\{T_{t,\pm}\}_{t\geq 0}$ satisfy \begin{itemize}
\item[(i)] $\lVert T_{t,\pm}f \rVert_{\ell^p}\leq \lVert f\rVert_{\ell^p}$ and $T_{0,\pm}f=f.$
\item[(ii)] $T_{t,\pm}T_{s,\pm}f=T_{t+s,\pm}f.$
\item[(iii)] $\lim_{t\to 0}T_{t,\pm}f=f$ on $\ell^p(\Z).$
\item[(iv)] $T_{t,\pm}f\geq 0$ if $f\geq 0.$
\item[(v)] $T_{t,\pm}1=1.$
\item[(vi)] $T^*_{t,\pm}=T_{t,\mp}$ on $\ell^2(\Z).$
\end{itemize}
\end{proposition}

\begin{proof} $T_{0,\pm}f=f$ by definition.
We prove the rest of the results for $T_{t,+}$ (the proof is analogous for $T_{t,-}$). Let $f\in \ell^p(\Z)$ for $1\leq p<\infty.$ By Minkowski's inequality  $$\lVert T_{t,+}f\rVert_{\ell^p}\leq e^{-t}\sum_{j=0}^{\infty}\frac{t^j}{j!}\biggl(\sum_{n\in\Z}|f(n+j)|^p\biggr)^{\frac{1}{p}}=\lVert f\rVert_{\ell^p}.$$ For $p=\infty,$ is analogous.
In order to prove (ii) we use the Newton's binomial, \begin{eqnarray*}
T_{t,+}T_{s,+}f(n)&=&e^{-(t+s)}\displaystyle\sum_{j=0}^{\infty}\frac{t^j}{j!}\sum_{h=0}^{\infty}\frac{s^h}{h!}f(n+j+h)=e^{-(t+s)}\displaystyle\sum_{j=0}^{\infty}\frac{t^j}{j!}\sum_{u=j}^{\infty}\frac{s^{u-j}}{(u-j)!}f(n+u)\\
&=&e^{-(t+s)}\displaystyle\sum_{u=0}^{\infty}\frac{f(n+u)}{u!}\sum_{j=0}^{u}\binom{u}{j}t^js^{u-j}=T_{t+s,+}f(n).
\end{eqnarray*}

For (iii)  we use that $f(n+j)-f(n)=0$ for $j=0,$ and Minkowski's inequality to get    \begin{eqnarray*}
\lVert T_{t,\pm}f-f\rVert_{\ell^p}&=&\biggl(\displaystyle\sum_{n\in\Z}\left|e^{-t}\sum_{j=1}^{\infty}\frac{t^j}{j!}(f(n+j)-f(n))\right|^p\biggr)^{\frac{1}{p}} \\
&\leq& e^{-t}\displaystyle\sum_{j=1}^{\infty}\biggl( \sum_{n\in\Z}\left|\frac{t^j}{j!}(f(n+j)-f(n))\right|^p \biggr)^{\frac{1}{p}} \\
&\leq&2e^{-t}\lVert f\rVert_{\ell^p}\displaystyle\sum_{j=1}^{\infty}\frac{t^j}{j!}\to 0
\end{eqnarray*} as $t\to 0.$ For $p=\infty$ is similar. We leave the verification of  (iv),  (v) and (vi) to the reader.\end{proof}

\begin{remark}\label{semigrupos2}
By Lemma \ref{'lema27'} and Proposition \ref{'prop18'} (i)-(iii), observe that the one-parameter operator families $\{T_{t,\pm}\}_{t\geq 0}$ are uniformly bounded $C_0$-semigroups on $\ell^p(\Z)$ for $1\leq p\leq \infty,$ generated by $-\delta_{\rm right/left},$ in the sense of the operator theory, see \cite{ABHN}. Furthermore, it is easy to see that the domain of $\delta_{\rm right/left}$ on $\ell^p(\Z)$ is the whole space.
\end{remark}

Now we shall see that the function $u(n,t) = T_{t,+}f(n)= \sum_{j\ge 0} G_t(j) f(n+j)$ is a solution of (\ref{eq1}). Observe that
$$\frac{\partial}{\partial t}G_t(0)=-G_t(0),\quad \hbox{and } \, \frac{\partial}{\partial t}(G_t(j))=-G_t(j)+G_t(j-1),\quad j\geq 1.$$
On the other hand, for any $A>0$,  $\sum_j \frac{A^j}{ j!} \le C_A < \infty$, hence we can differentiate term by term the series and we have
\begin{align*}
\partial_tT_{t,+}f(n)&=-G_t(0)f(n)+\sum_{j=1}^\infty(-G_t(j)+G_t(j-1))f(n+ j)\\
&=-\sum_{j\geq 0}G_t(j)f(n+ j) +\sum_{j\geq 0}G_t(j)f(n+(j+1))=-\delta_{\rm right} T_tf(n).
\end{align*}
The proof of the result for $\delta_{\rm left}$ is analogous.

 Once we have enclosed the fractional differences into the frame of semigroups, we take advantage of the method to highlight some properties and interesting results of these operators. We recall to the reader the following Gamma function formulas for an operator $L$.
\begin{equation}\label{'eq19'}
L^\alpha=\frac{1}{\Gamma(-\alpha)}\int_0^\infty(e^{-t L}-1)\frac{dt}{t^{1+\alpha}} \quad \hbox{  and  }\quad
L^{-\alpha}=\frac{1}{\Gamma(\alpha)}\int_0^\infty e^{-t L}\frac{dt}{t^{1-\alpha}},
\end{equation}
where $e^{-t L}$ is the associated semigroup, see \cite{Bernardis,  Stein, ST,Yosida}.
 In particular, we can define
\begin{align*}
(\delta_{\rm right})^\alpha f(n)=\frac{1}{\Gamma(-\alpha)}\int_0^\infty \frac{T_{t,+}f(n) -f(n)}{t^{1+\alpha}}dt, \  0< \alpha < 1,
\end{align*}
and
\begin{align*}
(\delta_{\rm right})^{-\alpha} f(n)=\frac{1}{\Gamma(\alpha)}\int_0^\infty \frac{T_{t,+}f(n)}{t^{1-\alpha}}dt, \  0< \alpha < 1,
\end{align*}
and the corresponding formula for $(\delta_{\rm left})^\alpha, \, -1< \alpha <1.$
We shall prove that this definition of $(\delta_{\rm right})^\alpha$ coincides with the definition (\ref{Chapman}) given by Chapman in 1911.

Along this paper, for $\alpha\in\R$, we denote $\Lambda^\alpha(m)=\binom{-\alpha-1+m}{m}=(-1)^m \binom{\alpha}{m} $, $m\in\N_0$.
The sequences $\{\Lambda^{\alpha}(n)\}_{n\in \N_0}$ have been studied in a more general setting in \cite{ALMV, Zygmund}. Here we highlight some properties of this kernel. For $0<\alpha<1,$ we have $\sum_{n=0}^{\infty}\Lambda^{\alpha}(n)=0,$ $\sum_{n=1}^{\infty}\Lambda^{\alpha}(n)=-1.$ In addition, we can observe that $\Lambda^{\alpha}(n)<0$ for $n\in\N,$ and \begin{equation}\label{asym}\Lambda^{\alpha}(n)=\frac{1}{n^{\alpha+1}\Gamma(-\alpha)}\left(1+O\left({1\over n}\right)\right),\end{equation}see \cite{Samko,Zygmund}.

Then
\begin{align*}
(\delta_{\rm right})^\alpha f(n)&=\frac{1}{\Gamma(-\alpha)}\int_0^\infty \frac{e^{-t}\sum_{j=0}^\infty \frac{t^j}{j!}(f(n+j)-f(n))}{t^{1+\alpha}}dt=\sum_{j=0}^\infty(f(n+j)-f(n))\int_0^\infty\frac{e^{-t}t^{j-\alpha}}{j!\Gamma(-\alpha)}\frac{dt}{t}\\
&=\sum_{j=0}^\infty(f(n+j)-f(n)) \frac{\Gamma(-\alpha+j)}{\Gamma(-\alpha)j!}=\sum_{j=0}^\infty\Lambda^\alpha(j) f(n+j),
\end{align*}
where the interchange of the sum and the integral is justified because of the integral converges absolutely. In the last equality we have used that  $\sum_{j=0}^\infty \Lambda^\alpha(j) = 0$.
By a similar way we also get
$$(\delta_{\rm left})^\alpha=\sum_{j=0}^\infty\Lambda^\alpha(j) f(n-j).$$

\begin{remark}
The above expression of $(\delta_{\rm right/left})^\alpha$ coincides with the formula of the fractional powers of $\delta_{\rm right/left}$ as generators of uniformly bounded $C_0$-semigroups on $\ell^p(\Z),$ in the sense of Balakrishnan, see \cite[Chapter IX, Section 11]{Yosida}.
\end{remark}

\begin{remark}(Probabilistic interpretation). Let $u$ be a function defined on $\Z_h$ such that its progressive difference is zero, which is equivalent to write $$u(jh)=u((j+1)h),\quad j\in\Z.$$ This implies that the discrete function $u$ is constant. We can interpret the above identity as the movement of a particle that compulsorily jumps to the adjacent right point on the mesh. If now we suppose that $(\delta_{\rm right})^{\alpha}u=0,$ then $$u(jh)=-\displaystyle\sum_{n=1}^{\infty}\Lambda^{\alpha}(n)u((j+n)h).$$ Since $-\sum_{n=1}^{\infty}\Lambda^{\alpha}(n)=1,$ the fractional identity can describe the movement of a particle which is able to jump to the right points $j+n$ with probability $-\Lambda^{\alpha}(n).$ It is easy to see that we recover the first situation as $\alpha\to 1^{-}.$ If $\alpha\to 0^+,$ the particle tends to be still.
\end{remark}
\section{Properties from a PDE point of view}\label{PDE}

\setcounter{theorem}{0}
\setcounter{equation}{0}

In this section we study the fractional discrete differences, and their regularity on the discrete H\"older spaces.

In the previous section we have considered functions defined on $\Z.$ Now we work on $\Z_h=h\Z,$ for $h>0,$ since one of our main objectives will be to compare the fractional discrete differences on $\Z_h$ with the discretized continuous fractional derivatives as $h\to 0.$ This will be done in the next section. Let $u:\Z_h\to\R,$ the first order difference operators on $\Z_{h}$ are given by
$$
\delta_{\rm right}u(hn)=\frac{u(hn)-u(h(n+1))}{h},\quad \delta_{\rm left}u(hn)=\frac{u(hn)-u(h(n-1))}{h},\quad n\in\Z.
$$

Notice that $\{T_{\frac{t}{h},\pm}\}_{t\geq0}$ are the contraction semigroups on $\ell^{p}(\Z_h)$, $1\leq p\leq\infty$, generated by $-\delta_{\rm right}$ and $-\delta_{\rm left}.$ Then, by the results of the last sections  we can write
$$
(\delta_{\rm right})^{\alpha}u(nh)=\frac{1}{h^{\alpha}}\sum_{j=0}^{\infty}(u((n+j)h)-u(nh))\Lambda^\alpha(j).
$$

and
$$
(\delta_{\rm left})^{\alpha}u(nh)=\frac{1}{h^{\alpha}}\sum_{j=0}^{\infty}(u((n-j)h)-u(nh))\Lambda^\alpha(j).
$$
Now we shall prove the maximum and comparison principles for the fractional differences $(\delta_{\rm right})^{\alpha}$ and the uniqueness of the corresponding Dirichlet problems stated in the Introduction as Theorem  \ref{'teo22'} and
 Theorem \ref{'cor23'}. Observe that the statements and proofs for $\delta_{\rm left}$ are analogous.
\vspace{0.5cm}

{\it Proof of Theorem \ref{'teo22'}:}
Part (ii) is a straightforward consequence of (i). For (i) we write $(\delta_{\rm right})^{\alpha}u(j_0h)=\frac{1}{h^{\alpha}}\displaystyle\sum_{m\in\N_0}\Lambda^{\alpha}(m)u((j_0+m)h)=\frac{1}{h^{\alpha}}\displaystyle\sum_{m\in\N}\Lambda^{\alpha}(m)u((j_0+m)h)\leq0,$ since $\Lambda^{\alpha}(m)<0$ for all $m\in\N.$ Moreover, by the same argument, if $(\delta_{\rm right})^{\alpha}u(j_0h)=0$ then $u((j_0+m)h)=0$ for all $m\in\N.$
$\hfill {\Box}$
\vspace{0.5cm}

{\it Proof of Theorem \ref{'cor23'}:}
We prove part (i) by contradiction. We suppose that there exists $m\in [j_0,j_1)$ such that $u(mh)<0$ is the global minimum of $u$ in $[j_0,\infty)$. So $u(jh)-u(mh)\geq0$ for all $j\geq j_0,$ then, by the maximum principle, $(\delta_{\rm right})^{\alpha}(u-u(mh))(mh)=(\delta_{\rm right})^{\alpha}u(mh)\leq 0.$ On the one hand, if $(\delta_{\rm right})^{\alpha}u(mh)=0,$ then $u(jh)=u(mh)<0$ for all $j\geq m,$ which contradicts that $u(jh)=0$ for all $j\geq j_1.$ On the other hand, if $(\delta_{\rm right})^{\alpha}u(mh)<0,$ this contradicts the hypothesis on $f$.

For part (ii), we use again an argument of contradiction. We suppose that $\sup_{j\geq j_0}u(jh)$ is not attained in $[j_1,\infty).$ Then there exists $m\in[j_0,j_1)$ such that $u(mh)$ is the global maximum of $u$ in $[j_0,\infty).$ So $u(mh)-u(jh)\geq0$ for all $j\geq j_0,$ then, by the maximum principle, $(\delta_{\rm right})^{\alpha}(u(mh)-u)(mh)=-(\delta_{\rm right})^{\alpha}u(mh)\leq 0$, i.e, $(\delta_{\rm right})^{\alpha}u(mh)\geq 0$. If $(\delta_{\rm right})^{\alpha}u(mh)>0$, it contradicts that $(\delta_{\rm right})^{\alpha}u\leq 0$ in $[j_0,j_1)$ and if $(\delta_{\rm right})^{\alpha}u(mh)=0,$ then, by the maximum principle, $u(jh)=u(hm)$ for all $j\geq m,$ so the $\sup_{j\geq j_0}u(jh)$ is attained in $[j_1,\infty).$

Part (ii) implies part (iii) by taking $-u$. Finally, part (iv) is a consequence of (iii), and the uniqueness is a straightforward consequence.
$\hfill {\Box}$
\vspace{0.5cm}

The regularity results announced in the Introduction are based in the behavior of our fractional powers when acting  over some H\"older spaces that we define now. Following the notation in \cite{Ciaurri}, for $l,s\in\N_0,$ we denote $\delta_{\rm right, \rm left}^{l,s}:=(\delta_{\rm right})^{l}(\delta_{\rm left})^{s}.$

\begin{definition}\textrm{(\cite[Definition 2.1]{Ciaurri}).} Let $0<\beta\leq 1$ and $k\in\N_0.$ A function $u:\Z_{h} \to \R$ belongs to the discrete H\"older space $C_{h}^{k,\beta}$ if $$[\delta_{\rm right, \rm left}^{l,s}u]_{C_{h}^{0,\beta}}=\displaystyle\sup_{m\neq j}\frac{|\delta_{\rm right, \rm left}^{l,s}u(jh)-\delta_{\rm right, \rm left}^{l,s}u(hm)|}{h^{\beta}|j-m|^{\beta}}<\infty$$ for each pair $l,s\in\N_0$ such that $l+s=k.$ The norm in the spaces $C_{h}^{k,\beta}$ is given by $$\lVert u\rVert_{C_{h}^{k,\beta}}=\displaystyle\max_{l+s\leq k}\sup_{m\in\Z}|\delta_{\rm right, \rm left}^{l,s}u(mh)|+\max_{l+s=k}[\delta_{\rm right, \rm left}^{l,s}u]_{C_{h}^{0,\beta}}.$$
\end{definition}

For simplicity, we only write the following theorem for $(\delta_{\rm right})^{\alpha}$ since it is analogous for $(\delta_{\rm left})^{\alpha}.$

\begin{theorem} \label{'teo24'}
Let $0< \beta\leq 1$ and $0<\alpha<1.$ \begin{itemize}
\item[(i)] Let $u\in C_{h}^{0,\beta}$ and $\alpha<\beta.$ Then $(\delta_{\rm right})^{\alpha}u\in C_{h}^{0,\beta-\alpha}$ and $$\lVert (\delta_{\rm right})^{\alpha}u \rVert_{C_{h}^{0,\beta-\alpha}}\leq C\lVert u \rVert_{C_{h}^{0,\beta}}.$$
\item[(ii)] Let $u\in C_{h}^{1,\beta}$ and $\alpha<\beta.$ Then $(\delta_{\rm right})^{\alpha}u\in C_{h}^{1,\beta-\alpha}$ and $$\lVert (\delta_{\rm right})^{\alpha}u \rVert_{C_{h}^{1,\beta-\alpha}}\leq C\lVert u \rVert_{C_{h}^{1,\beta}}.$$
\item[(iii)] Let $u\in C_{h}^{1,\beta}$ and $\alpha>\beta.$ Then $(\delta_{\rm right})^{\alpha}u\in C_{h}^{0,\beta-\alpha+1}$ and $$\lVert (\delta_{\rm right})^{\alpha}u \rVert_{C_{h}^{0,\beta-\alpha+1}}\leq C\lVert u \rVert_{C_{h}^{1,\beta}}.$$
\item[(iv)] Let $u\in C_{h}^{k,\beta}$ and assume that $k+\beta-\alpha$ is not an integer, with $\alpha<k+\beta.$ Then $(\delta_{\rm right})^{\alpha}u\in C_{h}^{l,s}$ where $l$ is the integer part of $k+\beta-\alpha$ and $s=k+\beta-\alpha-l$.
\end{itemize}
The positive constants $C$ are independent of $h$ and $u.$
\end{theorem}
\begin{proof}

Let $j,l\in\Z.$ We write $$|(\delta_{\rm right})^{\alpha}u(jh)-(\delta_{\rm right})^{\alpha}u(lh)|=\frac{1}{h^{\alpha}}|I_1+I_2|,$$ with $\displaystyle I_1=\displaystyle\sum_{1\leq m\leq |j-l|}(u((j+m)h)-u(jh)-u((l+m)h)+u(lh))\Lambda^{\alpha}(m),$ and \newline $\displaystyle I_2=\displaystyle\sum_{m> |j-l|}(u((j+m)h)-u(jh)-u((l+m)h)+u(lh))\Lambda^{\alpha}(m).$

To prove (i), note that $$|I_1|\leq C_{\alpha}h^{\beta}[u]_{C_h^{0,\beta}}\displaystyle\sum_{1\leq m\leq |j-l|}\frac{m^{\beta}}{m^{\alpha+1}}\leq C_{\alpha}h^{\beta}[u]_{C_h^{0,\beta}}|j-l|^{\beta-\alpha},$$ where we have used \eqref{asym}. For $I_2,$ observe that $$|u(jh)-u(lh))|,|u((j+m)h)-u((l+m)h)|\leq h^{\beta}[u]_{C_h^{0,\beta}}|j-l|^{\beta},$$ then $$|I_2|\leq C_{\alpha}h^{\beta}[u]_{C_h^{0,\beta}}|j-l|^{\beta}\displaystyle\sum_{m> |j-l|}\frac{1}{m^{\alpha+1}}\leq C_{\alpha}h^{\beta}[u]_{C_h^{0,\beta}}|j-l|^{\beta-\alpha},$$ using again \eqref{asym}.

Part (ii) is a straightforward consequence of (i). By definition, if $u\in C^{1,\beta}_h$, then $\delta_{\rm right}u\in C^{0,\beta}_h$, and as $\delta_{\rm right, \rm left}$ commutes with $(\delta_{\rm right})^{\alpha},$ by using (i) we get $\delta_{\rm right/left}(\delta_{\rm right})^{\alpha}u\in C_h^{0,\beta-\alpha}.$

For part (iii), suppose without loss of generality that $j>l$. We can write $\displaystyle u((j+m)h)-u(jh)=-h\displaystyle\sum_{p=0}^{m-1}\delta_{\rm right}u((j+p)h).$ Then
 \begin{align*}
|I_1|&\leq h\displaystyle\sum_{1\leq m\leq |j-l|}\sum_{p=0}^{m-1}|\delta_{\rm right}u((j+p)h)-\delta_{\rm right}u((l+p)h)|\Lambda^{\alpha}(m) \\
&\leq C_{\alpha}h^{\beta+1}\lVert u\rVert_{C_h^{1,\beta}}|j-l|^{\beta}\displaystyle\sum_{1\leq m\leq |j-l|}\frac{1}{m^{\alpha}}\leq  C_{\alpha}h^{\beta+1}\lVert u\rVert_{C_h^{1,\beta}}|j-l|^{\beta+1-\alpha}.
\end{align*}
To bound $I_2$ we write $\displaystyle u((j+m)h)-u((l+m)h)=-h\displaystyle\sum_{p=0}^{|j-l|-1}\delta_{\rm right}u((l+m+p)h),$ then  \begin{align*}
|I_2|&\leq h\displaystyle\sum_{m> |j-l|}\sum_{p=0}^{|j-l|-1}|\delta_{\rm right}u((l+m+p)h)-\delta_{\rm right}u((l+p)h)|\Lambda^{\alpha}(m) \\
&\leq C_{\alpha}h^{\beta+1}\lVert u\rVert_{C_h^{1,\beta}}|j-l|\displaystyle\sum_{m> |j-l|}\frac{m^{\beta}}{m^{\alpha}}\leq C_{\alpha}h^{\beta+1}\lVert u\rVert_{C_h^{1,\beta}}|j-l|^{\beta+1-\alpha}.
\end{align*}
Iterating parts (i), (ii) and (iii) we  get (iv).
\end{proof}

\section{Approximation of fractional derivatives in the line by discrete fractional derivatives. Marchaud and  Gr\"unwald-Letnikov fractional derivatives}
\setcounter{theorem}{0}
\setcounter{equation}{0}

Now we compare the fractional discrete differences and the discretized continuous fractional derivatives on H\"older spaces, and we estimate the error of the approximation on $\ell^{\infty}(\Z)$.

In \cite{Bernardis}, the fractional powers of the derivatives from the right and the left are considered, where $$D_{\text{right}}f(x)=\displaystyle\lim_{t\to 0^+}\frac{f(x)-f(x+t)}{t}\,  \hbox{  and  }  \, D_{\text{left}}f(x)=\displaystyle\lim_{t\to 0^+}\frac{f(x)-f(x-t)}{t}$$ are the continuous derivatives from the right and from the left, respectively. We recall that the fractional derivatives for $0<\alpha<1$ are given by $$(D_{\text{right/left}})^{\alpha}f(x)=\frac{1}{\Gamma(-\alpha)}\int_0^{\infty}\frac{f(x\pm t)-f(x)}{t^{1+\alpha}}\,dt,$$
for sufficiently smooth functions $f$.

Also, recall that a continuous real function $u$ belongs to the H\"older space $C^{k,\beta}(\R)$ if $u\in C^k(\R)$ and $$[u^{(k)}]_{C^{0,\beta}(\R)}=\displaystyle\sup_{x\neq y\in\R}\frac{|u^{(k)}(x)-u^{(k)}(y)|}{|x-y|^{\beta}}<\infty,$$ where $u^{(k)}$ denotes the $k$-th derivative of $u.$ The norm in the spaces $C^{k,\beta}(\R)$ is $$\lVert u\rVert_{C^{k,\beta}(\R)}=\displaystyle\sum_{l=0}^k \lVert u^{(l)}\rVert_{L^{\infty}(\R)}+[u^{(k)}]_{C^{0,\beta}(\R)}.$$

We only compare $(\delta_{\rm right})^{\alpha}$ with $(D_{\text{right}})^{\alpha}$ but the result and the proof are analogous for $(\delta_{\rm left})^{\alpha}$ and $(D_{\text{left}})^{\alpha}.$
To prove Theorem \ref{'teo26'} we need a previous lemma.

\begin{lemma} \label{'lema25'}Let $0<\alpha<1$ and $j\in \Z.$ We have \begin{equation}\label{lemma(1)}
\left|\frac{1}{\Gamma(-\alpha)}\int_{(j+m)h}^{(j+m+1)h}\frac{dt}{(t-jh)^{1+\alpha}}-\frac{\Lambda^{\alpha}(m)}{h^{\alpha}}\right|\leq \frac{C_{\alpha}}{h^{\alpha}m^{2+\alpha}},\quad m\in\N,
\end{equation} and \begin{equation}\label{lemma(2)}
\int_{(j+m)h}^{(j+m+1)h}\frac{dt}{(t-jh)^{1+\alpha}}\leq \frac{C_{\alpha}}{h^{\alpha}m^{1+\alpha}},\quad m\in\N.
\end{equation}
\end{lemma}
\begin{proof}
By doing the change of variable $t-jh=zh$ we have \begin{displaymath}\begin{array}{l}
\displaystyle\left|\frac{1}{\Gamma(-\alpha)}\int_{(j+m)h}^{(j+m+1)h}\frac{1}{(t-jh)^{1+\alpha}}\,dt-\frac{\Lambda^{\alpha}(m)}{h^{\alpha}}\right|=\left|\frac{1}{h^{\alpha}\Gamma(-\alpha)}\int_{m}^{m+1}\frac{dz}{z^{1+\alpha}}-\frac{\Lambda^{\alpha}(m)}{h^{\alpha}}\right| \\ \\
\leq \displaystyle\left|\frac{1}{h^{\alpha}\Gamma(-\alpha)}\int_{m}^{m+1}\biggl(\frac{1}{z^{1+\alpha}}-\frac{1}{m^{1+\alpha}}\biggr)dz\right|+\frac{1}{h^{\alpha}}\left|\frac{1}{\Gamma(-\alpha)m^{1+\alpha}}-\Lambda^{\alpha}(m)\right|.
\end{array}\end{displaymath}
The Mean Value Theorem implies
$$
\left|\int_{m}^{m+1}\biggl(\frac{1}{z^{1+\alpha}}-\frac{1}{m^{1+\alpha}}\biggr)dz\right|\leq C_{\alpha}\left|\int_{m}^{m+1}\frac{1}{m^{2+\alpha}}dz\right|\leq \frac{C_{\alpha}}{m^{2+\alpha}},$$
and
 $$\left|\frac{1}{\Gamma(-\alpha)m^{1+\alpha}}-\Lambda^{\alpha}(m)\right|\leq \frac{C_{\alpha}}{m^{2+\alpha}}
 $$
  by \eqref{asym}. So, \eqref{lemma(1)} is proved.
For \eqref{lemma(2)}, we have $$\int_{(j+m)h}^{(j+m+1)h}\frac{dt}{(t-jh)^{1+\alpha}}\leq \int_{(j+m)h}^{(j+m+1)h}\frac{dt}{(mh)^{1+\alpha}}\leq \frac{C_{\alpha}}{h^{\alpha}m^{1+\alpha}}.$$
\end{proof}

{\it Proof of Theorem \ref{'teo26'}:}

We suppose the hypothesis of part (i). Let $j\in\Z,$ then \begin{eqnarray*}
r_h ((D_{\text{right}})^{\alpha}u)(j)&=&\displaystyle \frac{1}{\Gamma(-\alpha)}\sum_{m\in\N_0}\int_{(j+m)h}^{(j+m+1)h}\frac{u(t)-u(jh)}{(t-jh)^{1+\alpha}}\,dt \\
&=&\frac{1}{\Gamma(-\alpha)}\biggl( \sum_{m\in\N_0}\int_{(j+m)h}^{(j+m+1)h}\frac{u(t)-u((j+m)h)}{(t-jh)^{1+\alpha}}\,dt \\
&&+\sum_{m\in\N}\int_{(j+m)h}^{(j+m+1)h}\frac{u((j+m)h)-u(jh)}{(t-jh)^{1+\alpha}}\,dt\biggr)\\
&=&\frac{1}{\Gamma(-\alpha)}(I_1+I_2).
\end{eqnarray*}
On the one hand, \begin{eqnarray*}
|I_1|&\leq & C[u]_{C^{0,\beta}(\R)}\sum_{m\in\N_0}\int_{(j+m)h}^{(j+m+1)h}\frac{|t-(j+m)h|^{\beta}}{(t-jh)^{1+\alpha}}\,dt \\
&\leq& C_{\alpha}h^{\beta-\alpha}[u]_{C^{0,\beta}(\R)}\left(1+\sum_{m\in\N}\frac{1}{m^{1+\alpha}}\right)= C_{\alpha}h^{\beta-\alpha}[u]_{C^{0,\beta}(\R)},
\end{eqnarray*}
where we have used \eqref{lemma(2)}. On the other hand we compare $I_2$ with $(\delta_{\rm right})^{\alpha}(r_h u)(j)$. By \eqref{lemma(1)}, \begin{displaymath}\begin{array}{l}
\displaystyle\left|\frac{I_2}{\Gamma(-\alpha)}-(\delta_{\rm right})^{\alpha}(r_h u)(j)\right|\\
\displaystyle\leq\sum_{m\in\N}|u((j+m)h)-u(jh)|\left|\frac{1}{\Gamma(-\alpha)}\int_{(j+m)h}^{(j+m+1)h}\frac{dt}{(t-jh)^{1+\alpha}}-\frac{\Lambda^{\alpha}(m)}{h^{\alpha}}\right|\\
\displaystyle\leq C_{\alpha} h^{\beta -\alpha}[u]_{C^{0,\beta}(\R)}\sum_{m\in\N}\frac{m^{\beta}}{m^{2+\alpha}}\leq C_{\alpha} h^{\beta -\alpha}[u]_{C^{0,\beta}(\R)}.
\end{array}\end{displaymath}

For (ii), observe that $\delta_{\rm right}$ commutates with $(\delta_{\rm right})^{\alpha}$ and $\frac{d}{dx}$ with $(D_{\text{right}})^{\alpha}.$ Then we write \begin{eqnarray*}
\lVert -\delta_{\rm right}(\delta_{\rm right})^{\alpha}(r_h u)-r_h (\frac{d}{dx}(D_{\text{right}})^{\alpha}u)\rVert_{\ell^{\infty}}&\leq& \lVert (\delta_{\rm right})^{\alpha}(-\delta_{\rm right})(r_h u)-(\delta_{\rm right})^{\alpha}(r_h\frac{d}{dx}u)\rVert_{\ell^{\infty}}\\
&&+\lVert (\delta_{\rm right})^{\alpha}(r_h\frac{d}{dx}u)-r_h ((D_{\text{right}})^{\alpha}(\frac{d}{dx}u))\rVert_{\ell^{\infty}}.
\end{eqnarray*}
We apply the part (i) to the second term. Let $j\in\Z.$ For the first one, we apply the Mean Value Theorem and the fact that $\sum_{m=0}^\infty\Lambda^\alpha(m)=0$,
\begin{small}
\begin{align*}
&\left|(\delta_{\rm right})^{\alpha}(-\delta_{\rm right})(r_h u)(j)-(\delta_{\rm right})^{\alpha}(r_h\frac{d}{dx}u)(j)\right|=\\
&=\frac{1}{h^{\alpha}}\left|\displaystyle\sum_{m\in\N}\Lambda^{\alpha}(m)\biggl(\frac{u((j+m+1)h)-u((j+m)h)}{h}-u'((j+m)h)-\frac{u((j+1)h)-u(jh)}{h}+u'(jh)\biggr)\right|\\ &=\frac{1}{h^{\alpha}}\left|\displaystyle\sum_{m\in\N}\Lambda^{\alpha}(m)\biggl(u'(\xi_{j+m})-u'((j+m)h)-u'(\xi_{j})+u'(jh)\biggr)\right|\\ \\
&\leq\frac{C}{h^{\alpha}}[u']_{C^{0,\beta}(\R)}\displaystyle\sum_{m\in\N}|\Lambda^{\alpha}(m)|h^{\beta}\leq C_{\alpha}[u']_{C^{0,\beta}(\R)}h^{\alpha-\beta},
\end{align*}
\end{small}
where $\xi_j\in(jh,(j+1)h)$ and $\xi_{j+m}\in((j+m)h,(j+m+1)h).$
$\hfill {\Box}$
\vspace{0.5cm}

{\it Proof of Theorem \ref{GL-M}:}

Given $x\in\R$ and $h>0$, there exists a $j_0\in\Z$ such that $j_0h\leq x< j_0h+h$.
Then, we have
\begin{align*}
\left|(D_{\rm right})^{\alpha} f(x)-\frac{\Delta_{h,+}^\alpha}{h^\alpha}f(x)\right|&\leq \left|(D_{\rm right})^{\alpha} f(x)-(D_{\rm right})^{\alpha} f(j_0h)\right|+\left|(D_{\rm right})^{\alpha} f(j_0h)-\frac{\Delta_{h,+}^\alpha}{h^\alpha}f(j_0h)\right|\\
&+\left|\frac{\Delta_{h,+}^\alpha}{h^\alpha}f(j_0h)-\frac{\Delta_{h,+}^\alpha}{h^\alpha}f(x)\right|=I+II+III.
\end{align*}
On the one hand,
\begin{align*}
I&= \left|\frac{1}{\Gamma(-\alpha)}\int_0^{\infty}\frac{f(x+t)-f(x)}{t^{1+\alpha}}dt -\frac{1}{\Gamma(-\alpha)}\int_0^{\infty}\frac{f(j_0h+t)-f(j_0h)}{t^{1+\alpha}}dt\right|\\
&\leq \left| \frac{1}{\Gamma(-\alpha)}\int_0^h\frac{f(x+t)-f(x)}{t^{1+\alpha}}dt\right|+\left| \frac{1}{\Gamma(-\alpha)}\int_0^h\frac{f(j_0h+t)-f(j_0h)}{t^{1+\alpha}}dt\right|\\
&+\left| \frac{1}{\Gamma(-\alpha)}\int_h^\infty\frac{f(x+t)-f(j_0h+t)}{t^{1+\alpha}}dt\right|+\left| \frac{1}{\Gamma(-\alpha)}\int_h^\infty\frac{f(j_0h)-f(x)}{t^{1+\alpha}}dt\right|\\
&\leq C_{\alpha}[f]_{C^{0,\beta}(\R)}\left(\int_0^h\frac{t^\beta}{t^{1+\alpha}}dt+\int_h^\infty\frac{h^\beta}{t^{1+\alpha}}dt\right)=C_{\alpha,\beta}[f]_{C^{0,\beta}(\R)}h^{\beta-\alpha}.
\end{align*}

By using Theorem \ref{'teo26'} we obtain
$$
II\leq C_{\alpha}[f]_{C^{0,\beta}(\R)}h^{\beta-\alpha}.
$$

On the other hand, as $\sum_{k=0}^\infty\left|\binom{\alpha}{k}\right|\leq C_{\alpha}$, we obtain

\begin{align*}
III&\leq \left|\frac{1}{h^\alpha}\sum_{k=0}^\infty (-1)^k\binom{\alpha}{k}(f(j_0h+kh)-f(x+kh)) \right|\leq\frac{C[f]_{C^{0,\beta}(\R)} }{h^\alpha}\sum_{k=0}^\infty\left|\binom{\alpha}{k}\right| h^\beta\\
&\leq C_{\alpha}[f]_{C^{0,\beta}(\R)} h^{\beta-\alpha},
\end{align*}
so the result follows.

The proof is analogous for $(D_{\text{left}})^{\alpha}$ and $\frac{\Delta_{h,-}^\alpha}{h^\alpha}.$
$\hfill {\Box}$

\section{Extension problem}\label{Sextension}
\setcounter{theorem}{0}
\setcounter{equation}{0}

Theorem \ref{gale} is a straightforward consequence of Remark \ref{semigrupos2} and \cite[Theorem 1.1]{Gale}.  The formula \eqref{solution} provides an explicit expression in our case, that is,
\begin{align}\label{P.K.}
 u(t,n)=P_{t,\pm}^\gamma f(n)&=\frac{t^{2\gamma}}{4^\gamma \Gamma(\gamma)} \int_0^\infty e^{-\frac{t^2}{4s}} T_{s,\pm}f(n) \frac{ds}{s^{1+\gamma}}=\displaystyle\sum_{j=0}^{\infty}\frac{t^{2\gamma} f(n\pm j)}{4^\gamma\Gamma(\gamma)j!}\int_0^{\infty}e^{-s-t^2/4s}s^{j-\gamma}\frac{ds}{s} \\
&=\displaystyle\sum_{j=0}^{\infty} \frac{t^{j+\gamma}}{2^{j+\gamma-1}\Gamma(\gamma)j!}K_{j-\gamma}(t)f(n\pm j) \nonumber
\end{align}
where we have used the identity \cite[2.3.16.1, p. 344]{Prudnikov}. The function $K_{\nu}$ is the \textit{Macdonald's function} (also called \textit{modified Bessel function of the third type}) defined in \cite[Section 5.7, p. 108]{Lebedev}.
By completeness we prove that the previous formula solves \eqref{extension} pointwise for $t\in(0,\infty)$.
We shall use the following identities   (see \cite[Section 5.7]{Lebedev}):\begin{equation*}\label{Mac2}K_{\nu+1}(t)=\frac{2\nu}{t}K_\nu(t)+K_{\nu-1}(t),\   \quad  \frac{\partial}{dt}(t^{\nu}K_{\nu}(t))=-t^{\nu}K_{\nu-1}(t),\  \forall\nu\in\mathbb{R}.\end{equation*}
We have
\begin{align*}
\partial_t P^\gamma_{t,+}f(n)
&=\sum_{j=0}^\infty\frac{f(n+j)}{ 2^{j+\gamma-1}\Gamma(\gamma)j! }(2\gamma t^{j+\gamma-1}K_{j-\gamma}(t)-t^{\gamma+j}K_{j-\gamma-1}(t)),
\end{align*}
and
\begin{align*}
\partial_{tt}P_{t,+}^\gamma f(n)&=\sum_{j=0}^\infty\frac{f(n+j) t^{\gamma+j}}{2^{j+\gamma-1}\Gamma(\gamma)j! }\left(\frac{2\gamma(2\gamma-1)}{t^2}K_{j-\gamma}(t)-\frac{4\gamma+1}{t}K_{j-\gamma-1}(t)+K_{j-\gamma-2}(t)\right)\\
&=\sum_{j=0}^\infty\frac{f(n+j)t^{\gamma+j}}{2^{j+\gamma-1}\Gamma(\gamma)j!  }\left(\left(\frac{2\gamma(2\gamma-1)}{t^2}+1\right)K_{j-\gamma}(t)-\frac{2\gamma+2j-1}{t}K_{j-\gamma-1}(t) \right).
\end{align*}
Then, we obtain
\begin{align*}
(\partial^2_{tt} + \frac{1-2 \gamma}{t}\partial_t)P_{t,+}^\gamma f(n)&=\sum_{j=0}^\infty\frac{f(n+j)t^{\gamma+j}}{2^{j+\gamma-1}\Gamma(\gamma)j! }(K_{j-\gamma}-\frac{2j}{t}K_{j-\gamma-1}(t))\\
&=\delta_{\rm right}P_ {t,+}^\gamma f(n).
\end{align*}
The analogous result for $\delta_{\rm left}P_{t,-}^\gamma$ can be also proved by the same way.

Observe that in other words we have seen that  $P_{t,\pm}^\gamma f$, is precisely the Poisson semigroup associated to $\delta_{\rm right/left}$.

\section{Maximal operators. Littlewood-Paley functions}\label{secciondura}
\setcounter{theorem}{0}
\setcounter{equation}{0}
In this last section we shall consider the maximal functions and the Littlewood-Paley square functions associated to the heat and Poisson semigroups.
Our first observation is that both functions have bad behavior in the case of the heat semigroups. In fact we have the following.

\vspace{0.5 cm}
{\bf Claim $1$ }
{\it The maximal functions of the heat semigroups defined in (\ref{semigrupos}), that is,
$$
T^*_{\pm}f=\sup_{t\geq 0}|T_{t,\pm}f|,
$$ are not bounded from $\ell^p(\Z) $ into itself for $ 1\le p\le2.$}

In fact, let $f(0)=1$ and $f(j)=0$ for $j\neq 0,$ then $T_{t,-}f(n)=0$ for $n<0$ and $T_{t,-}f(n)=e^{-t}\frac{t^{n}}{n!}$ for $n\geq 0.$ The maximum of the function $e^{-t}\frac{t^{n}}{n!}$ is $e^{-n}\frac{n^{n}}{n!},$ and, by Stirling's formula, it behaves asymptotically like $\frac{1}{\sqrt{n}}$ as $n\to \infty.$

\vspace{0.5 cm}

{\bf Claim $2$ }
{\it The Littlewood-Paley functions of the heat semigroups defined by
$$
\displaystyle \mathfrak{g}_{\pm}(f)=\biggl( \int_0^{\infty}|t\partial_t T_{t,\pm}f|^2 \frac{dt}{t}\biggr)^{1/2},$$
 are not bounded from $\ell^2(\Z) $ into itself.}

Observe that  $\widehat{T_{t,+}f}(\theta)=e^{-t(1-e^{i\theta})}\hat{f}(\theta),$ and then it follows
$$\widehat{\partial_t T_{t,+}f}(\theta)=(e^{i\theta}-1)e^{-t(1-e^{i\theta})}\hat{f}(\theta).$$
Let $f\in \ell^2(\Z)$ such that $\int_0^{2\pi}\left(\frac{|\hat{f}(\theta)|}{\theta}\right)^2 d\theta = \infty.$
Hence, by Plancherel's and Fubini's Theorems, we have
\begin{align*}
\|\mathfrak{g}_+(f)\|^2_{\ell^2}&=\displaystyle\sum_{n\in\mathbb{Z}}\int_0^\infty |t\partial_t T_{t,+}f(n)|^2\frac{dt}{t}=\frac{1}{2\pi}\int_0^\infty\int_{\mathbb{T}}|\widehat{t\partial_t T_{t,+}f}(\theta)|^2d\theta\frac{dt}{t}\\
&=\frac{1}{2\pi}\int_0^\infty\int_{\mathbb{T}}|t(e^{i\theta}-1)e^{-t(1-e^{i\theta})}|^2|\hat{f}(\theta)|^2d\theta\frac{dt}{t}\\
&=\frac{1}{2\pi}\int_0^{2\pi}|\hat{f}(\theta)|^2\int_0^\infty t^2|e^{i\theta}-1|^2e^{-2t\Re(1-e^{i\theta})}\frac{dt}{t}d\theta\\
&=\frac{1}{\pi}\int_0^{2\pi}|\hat{f}(\theta)|^2\int_0^\infty t^2(1-\cos \theta)e^{-2t(1-\cos\theta)}\frac{dt}{t}d\theta\\
&\underbrace{=}_{\substack{2(1-\cos\theta)t=u}}\frac{1}{4\pi}\int_0^{2\pi}\frac{|\hat{f}(\theta)|^2}{1-\cos\theta}\int_0^\infty u e^{-u}du\,d\theta=\frac{1}{4\pi}\int_0^{2\pi}\frac{|\hat{f}(\theta)|^2}{1-\cos\theta}d\theta,
\end{align*}
and this integral does not converge.

However, the behavior of the Poisson semigroups is suitable with the classical results in Harmonic analysis. Consider the maximal function associated to the Poisson semigroups
$$
P^{\gamma,*}_{\pm}f=\sup_{t\geq 0}|P_{t,\pm}^{\gamma}f|,\quad 0<\gamma<1,
$$
and the Littlewood-Paley square functions
$$\displaystyle g^{\gamma}_{\pm}(f)=\biggl( \int_0^{\infty}|t\partial_t P^{\gamma}_{t,\pm}f |^2 \frac{dt}{t}\biggr)^{1/2}, \quad  0<\gamma<1.$$
%

To prove Theorem \ref{'teo20'} for these operators, the tool that we shall use is the vector valued Calder\'on-Zygmund Theory in spaces of homogeneous type, more specifically  in the particular case of the integers with the natural distance $d(n,m) = |n-m|$ and measure $\mu(n) =1.$
Given a Banach space $E$, we denote by $\ell^p_E(\Z)$, $1\leq p\leq\infty$, the space of $E$-valued
functions $f$ defined on $\Z$ such that $\|f\|_E$ belongs to $\ell^p(\Z,d,\mu)$.
\begin{definition}[Vector-valued (convolution) Calder\'on-Zygmund operator on $(\Z,d,\mu)$]\label{CZ}
We say that a linear  operator $\mathfrak{T}$ on
the space  $(\Z,d,\mu)$ is a Calder\'on-Zygmund operator if it satisfies the following conditions.
\begin{itemize}
\item [(I)] There exists $1\le p_0 \le\infty$ such that $\mathfrak{T}$ is bounded from
$\ell^{p_0}(\Z)$ into $\ell_E^{p_0}(\Z)$.
\item[(II)] For  bounded  functions $f$ with compact support, $\mathfrak{T}f$ can be represented as
\begin{equation*}
\mathfrak{T}f(n) = \sum_{j\in\Z} K(j)f(j+n),
\end{equation*}
where $K(j) \in \mathcal{L} (\R,E)$ is the space of bounded linear operator from $\R$ to $E$, and satisfies
\begin{itemize}
\item[(II.1)]  $\displaystyle \|K(j) \|_{\mathcal{L} (\R,E)} \le  \frac{C}{|j|}$, for every $j\neq 0$;
\item[(II.2)] $\displaystyle   \|K(j) - K(j_0) \|_{\mathcal{L}(\R,E)}
\le C\frac{|j-j_0|}{|j_0|^2},$
 whenever $|j_0|>2|j-j_0|,  \quad j_0\neq 0;$
\end{itemize}
for some constant $C>0$.
\end{itemize}
\end{definition}
The Calder\'on--Zygmund theorem says that if $\mathfrak{T}$ is a Calder\'on--Zygmund operator on $(\Z,d,\mu)$ as above
then $\mathfrak{T}$ is bounded from $\ell^{p}(\Z)$ into $\ell_E^{p}(\Z)$, for any  $1<p\le\infty$, and it is also of
weak type $(1,1)$.
For full details see \cite{MST, RRT, RT}.

Now we ready to prove the Theorem \ref{'teo20'}. We shall proof only the cases associated to $P_{t,+}^{\gamma}.$ For $P_{t,-}^{\gamma}.$ the proof is analogous.

{\it Proof of Theorem \ref{'teo20'}.}

{\bf Case 1. The maximal function.}

For convenience, we will write  $P_{t,+}^\gamma f(n)=\sum_{j=0}^\infty P_{t}^\gamma(j) f(n+ j),$ where
   $$P_{t}^\gamma(j)=\frac{t^{2\gamma}}{4^\gamma\Gamma(\gamma)j!}\int_0^{\infty}e^{-s-t^2/4s}s^{j-\gamma}\frac{ds}{s}
   \, , \, j\in\N_0.$$
Consider the vector-valued operator
$$\mathfrak{T}f( n) = \Big\{ \sum_{j\in\Z} P_{t}^\gamma(j)  f(n+ j)\Big\}_{t\ge 0}=
\sum_{j\in\Z} \Big\{P_{t}^\gamma(j)  \Big\}_{t\ge0}f( n+j),$$
where we have assumed $P_t^\gamma(j)=0 $ for $j<0$, $t\geq 0$.
The operator $\mathfrak{T}$ satisfies
$ \mathfrak{T}:  \ell^\infty (\Z)   \longrightarrow  \ell^\infty_{L^\infty} (\Z ) .$
In fact
\begin{multline*}
\|\mathfrak{T}f\|_{\ell^\infty_{L^\infty}(\Z)}=\sup_{n\in\Z}\sup_{t\geq 0}|P_{t,+}^\gamma f(n)|
\leq C_{\gamma}\|f\|_\infty\sup_{t\geq 0}\sum_{j=0}^\infty \frac{1}{j!} \int_0^\infty e^{-u}e^{-\frac{t^2}{4u}}u^j\left(\frac{t^2}{u}\right)^\gamma\frac{du}{u} \\
 = C_\gamma\|f\|_\infty\sup_{t\geq 0} \int_0^\infty e^{-\frac{t^2}{4u}}\left(\frac{t^2}{u}\right)^\gamma\frac{du}{u}\underbrace{=}_{\substack{\frac{t^2}{4u}=v}} C_\gamma\|f\|_\infty \sup_{t\geq 0}\int_0^\infty e^{-v}\left(4v\right)^\gamma\frac{dv}{v}<\infty,
\end{multline*}
where we have applied Fubini's Theorem.

Moreover the kernel $\{P_{t}^\gamma(j) \}_{t\geq 0}$ satisfies
\begin{align*}
\|P_{t}^\gamma(j)\|_{L^\infty}&=\sup_{t\geq 0}\frac{1}{4^{\gamma}\Gamma(\gamma)j!}\int_0^\infty e^{-u}e^{-\frac{t^2}{4u}}\left(\frac{t^2}{u}\right)^\gamma u^{j}\frac{du}{u}\leq\sup_{t\geq 0}\frac{C_{\gamma}}{j!}\int_0^\infty e^{-u}u^{j}\frac{du}{u}=C_{\gamma}\frac{1}{j},
\end{align*}
 for $j>0$ and $\|P_{t}^\gamma(j)\|_{L^\infty}=0 \le \frac{C_\gamma}{|j|}$ for $j<0,$ where we have used that the function $g(u)=e^{-\frac{t^2}{4u}}\left(\frac{t^2}{u}\right)^\gamma$ reaches its maximum at $u=t^2/4\gamma.$

Regarding (II.2), it is easy to see that it is enough to prove that for each $j\in \N$,
$$
\sup_{t\geq 0}|P_t^\gamma(j)-P_t^\gamma(j+1)|\leq \frac{C_\gamma}{j^2}.
$$
If $j\geq 1$, then for all $t\geq 0$ we have

\begin{align*}
|P_t^\gamma(j)-P_t^\gamma(j+1)|&=\frac{1}{4^{\gamma}\Gamma(\gamma)j!}\left|\int_0^\infty e^{-u}e^{-\frac{t^2}{4u}}\left(\frac{t^2}{u}\right)^\gamma u^{j}\left(1-\frac{u}{j+1}\right)\frac{du}{u}\right|\\
&=\frac{1}{4^{\gamma}\Gamma(\gamma)(j+1)!}\left|\int_0^\infty \partial_u(e^{-u}u^{j+1})e^{-\frac{t^2}{4u}}\left(\frac{t^2}{u}\right)^\gamma \frac{du}{u}\right|\\
&=\frac{1}{4^{\gamma}\Gamma(\gamma)(j+1)!}\left|\int_0^\infty e^{-u}u^{j+1}\partial_u\left(e^{-\frac{t^2}{4u}}\left(\frac{t^2}{u}\right)^\gamma\frac{1}{u}\right) du\right|,\\
\end{align*}
where we have used integration by parts.
As $ \left|\partial_u\left(e^{-\frac{t^2}{4u}}\left(\frac{t^2}{u}\right)^\gamma\frac{1}{u}\right)\right| \leq \frac{C_{\gamma}}{u^2},$ then $$|P_t^\gamma(j)-P_t^\gamma(j+1)|\leq \frac{C_{\gamma}}{(j+1)!}\int_0^\infty e^{-u}u^{j-1} du=\frac{C_{\gamma}\Gamma(j)}{(j+1)!}=\frac{C_{\gamma}}{j^2}.$$
Finally  since  $\|\mathfrak{T}f(n) \|_{L^\infty} = P^{\gamma,*}_{+}f(n)$  the result follows for the maximal operator, by choosing $p_0= \infty,$  and $E= L^\infty(\R_+).$

{\bf Case 2.  Littlewood-Paley functions}
Consider the vector-valued operator
$$\mathfrak{T}f(n) = \Big\{ \sum_{j\in\Z} t\partial_t P_t^\gamma(j) f(n+j)\Big\}_{t\ge0} = \sum_{j\in\Z} \Big\{t\partial_t P_t^\gamma(j) \Big\}_{t\ge0}f(n+j). $$
In this case we write
$$P_{t}^\gamma(j)=\frac{t^{2\gamma}}{4^\gamma\Gamma(\gamma)j!}\int_0^{\infty}e^{-s-t^2/4s}s^{j-\gamma}\frac{ds}{s}= \frac{1}{\Gamma(\gamma)}\int_0^{\infty}e^{-r} G_{t^2/4r}(j)\frac{dr}{r^{1-\gamma}}\, , \, j\in\N_0.$$  We have performed the change of variables  $r= \frac{t^2}{4s}$ and the sequence $G_t$ is defined in \eqref{semigrupos}.
Again we assume $P_t^\gamma(j) = 0$ for $j<0$, $t\geq 0$. The operator  $\mathfrak{T}$ is bounded from $\ell^2(\Z)$ into $\ell^2_{L^2((0,\infty), \frac{dt}{t})}$.  In fact
\begin{align*}
\|\mathfrak{T}f\|^2_{\ell^2_{L^2((0,\infty), \frac{dt}{t})}}&=\frac{1}{2\pi}\int_0^\infty\int_{\mathbb{T}}|\widehat{t\partial_t P^{\gamma}_{t,+}f}(\theta)|^2d\theta\frac{dt}{t}  = \frac{1}{2\pi}\int_{\mathbb{T}}\int_0^\infty|\widehat{t\partial_t P^{\gamma}_{t}}(\theta)|^2 \frac{dt}{t}|\hat{f}(-\theta)|^2 d \theta .
\end{align*}
We shall see that $\int_0^\infty|\widehat{t\partial_t P^{\gamma}_{t}}(\theta)|^2 \frac{dt}{t}<C_{\gamma},$ where $C_{\gamma}>0$ does not depend on $\theta.$ Observe that
$$ \widehat{t\partial_t P^{\gamma}_{t}}(\theta)   = \frac{1}{\Gamma(\gamma)}\int_0^{\infty}e^{-r} \frac{t^2}{2r} (e^{i\theta}-1)
e^{-\frac{t^2}{4r}(1-e^{i\theta})}\frac{dr}{r^{1-\gamma}}.$$

Notice that if $\theta=0,2\pi$, the statement is trivial, so we have to consider three cases: $0<\theta< \frac{\pi}{4}$, $\frac{\pi}{4}\le\theta\le \frac{7\pi}{4}$ and $\frac{7\pi}{4}<\theta< 2\pi$.

If $0<\theta< \frac{\pi}{4}$, we define $z_0=t(1-e^{i\theta})^{1/2},$ with $\varphi_0=\arg z_0\in (-\pi/4,\pi/4).$ Then,
 \begin{eqnarray*}\int_0^\infty|\widehat{t\partial_t P^{\gamma}_{t}}(\theta)|^2 \frac{dt}{t}&=& C_\gamma\int_0^\infty \Big|\frac{z_0^2}{2}\int_0^{\infty}e^{-r}e^{-z_0^2/(4r)}\frac{dr}{r^{2-\gamma}}\Big|^2 \frac{dt}{t}.\\
\end{eqnarray*}
By applying Cauchy's Theorem, for $\varepsilon,R>0$, we get
$$
\int_\varepsilon^{R}e^{-r}e^{-z_0^2/(4r)}\frac{dr}{r^{2-\gamma}}=\biggr(\int_{\Gamma_\varepsilon}-\int_{\Gamma_{R}}+\int_{\omega_\varepsilon}^{w_R}\biggl)e^{-\omega}e^{-z_0^2/(4\omega)}\frac{d\omega}{\omega^{2-\gamma}},
$$
where $$\Gamma=\{ \omega=sz_0\ | \ 0\leq s\leq\infty \},\ \ \Gamma_{\varepsilon}=\{ \omega=\varepsilon e^{i\varphi}\ | \varphi\in[0,\varphi_0] \},\ \ \Gamma_{R}=\{ \omega=R e^{i\varphi}\ |\varphi\in[0,\varphi_0] \},$$ and $\omega_{\varepsilon}=\varepsilon z_0,\, \omega_{R}=R z_0,$ see the next figure.

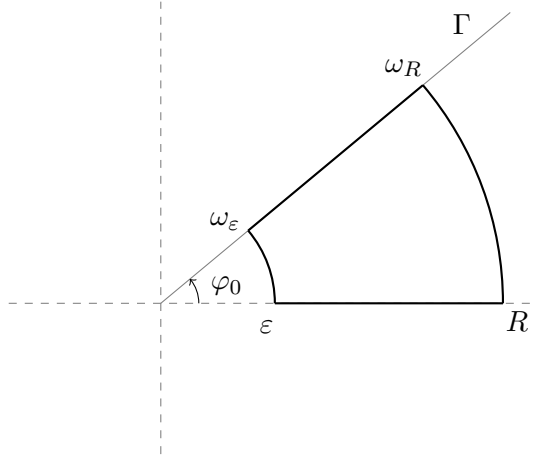
\vspace*{1.0cm}

\begin{center}

\begin{tikzpicture}
    \draw[dashed,gray](-2,0)--(5,0);
    \draw[dashed,gray](0,-2)--(0,4);

    \draw[gray](0:0cm)--(40:6cm);
    \path (0,0)++(15:0.9cm) node{$\varphi_0$};
    \draw[->] (0.5,0)arc(0:40:0.5cm);

    \draw[black, thick](40:1.5cm)--(40:4.5cm);
    \draw[black, thick] (1.5,0)arc(0:40:1.5cm);
    \draw[black, thick] (4.5,0)arc(0:40:4.5cm);
    \draw[black,thick](0:1.5cm)--(0:4.5cm);

   \draw[thick](65:1.2cm)node[right]{$\omega_{\varepsilon}$};
   \draw[thick](48:4.2cm)node[right]{$\omega_{R}$};
   \draw[thick](-15:1.2cm)node[right]{$\varepsilon$};
   \draw[thick](-3:4.4cm)node[right]{$R$};



    \draw[thick](3.7,3.7)node[right]{$\Gamma$};
    %
\end{tikzpicture}

\end{center}

\vspace*{0.5cm}

Notice that $$
\displaystyle \Big| \int_{\Gamma_\varepsilon}e^{-\omega}e^{-z_0^2/(4\omega)}\frac{d\omega}{\omega^{2-\gamma}} \Big|\leq \int_{[0,\varphi_0] }|e^{-\varepsilon e^{i\varphi}}e^{-\frac{z_0^2}{4\varepsilon}e^{-i\varphi}}|\frac{d\varphi}{\varepsilon^{1-\gamma}}
\leq \frac{|\varphi_0|e^{-\frac{|z_0|^2\cos(2\varphi_0)}{4\varepsilon}}}{\varepsilon^{1-\gamma}}\to 0, \, \text{as} \, \varepsilon \to 0,
$$ since $e^{-\varepsilon \cos\varphi}\leq 1$ and it can be checked that $\Re(\frac{z_0^2}{4\varepsilon} e^{-i\varphi})=\frac{|z_0|^2}{4\varepsilon}\cos(2\varphi_0-\varphi)\geq \frac{|z_0|^2}{4\varepsilon}\cos(2\varphi_0)$.

Similarly, we get \begin{displaymath}\begin{array}{l}
\displaystyle \Big| \int_{\Gamma_R}e^{-\omega}e^{-z_0^2/(4\omega)}\frac{d\omega}{\omega^{2-\gamma}} \Big|\leq
 \frac{|\varphi_0|e^{-\frac{|z_0|\cos(2\varphi_0)}{4 R}}}{R^{1-\gamma}}\to 0, \, \text{as} \,R \to \infty.
\end{array}\end{displaymath}
Hence
$$
\frac{z_0^2}{2}\int_0^{\infty}e^{-r}e^{-z_0^2/(4r)}\frac{dr}{r^{2-\gamma}}=\frac{z_0^2}{2}\int_{\Gamma}e^{-\omega}e^{-z_0^2/(4\omega)}\frac{d\omega}{\omega^{2-\gamma}}=\frac{z_0^{1+\gamma}}{2}\int_0^\infty e^{-z_0 s}e^{-z_0/(4s)}\frac{ds}{s^{2-\gamma}}
$$
 and
\begin{eqnarray*}\int_0^\infty|\widehat{t\partial_t P^{\gamma}_{t}}(\theta)|^2 \frac{dt}{t}&=&C_\gamma\int_0^{\infty} \Big|\frac{z_0^{1+\gamma}}{2} \int_0^{\infty}e^{-z_0 s}e^{-z_0/(4s)}\frac{ds}{s^{2-\gamma}} \Big|^2 \frac{dt}{t}\\
&\leq& C_{\gamma} \int_0^{\infty} \biggl( t^{1+\gamma} \theta^{\frac{1+\gamma}{2}}\int_0^{\infty}e^{-ct\sqrt{\theta} \cos(\varphi_0)s}e^{-ct\sqrt{\theta}\cos(\varphi_0)/(4s)}\frac{ds}{s^{2-\gamma}} \biggr)^2 \frac{dt}{t}\\
&=&C_{\gamma}\int_0^{\infty} \biggl( t^{1+\gamma}\theta^{\frac{1+\gamma}{2}} K_{\gamma-1}(ct\sqrt{\theta}\cos(\varphi_0)) \biggr)^2 \frac{dt}{t},
\end{eqnarray*} where we have used that $|1-e^{i\theta}|\sim \theta$, for $\theta$ being closed to $0$. In the last identity  $K_{\nu}$ denotes  the Mcdonald's function, see \cite[2.3.16.1, p. 344]{Prudnikov}. The following properties  can be found in \cite[Section 5.7 and Section 5.16]{Lebedev}
 \begin{equation}\label{Mac}K_{-\nu}(z)=K_{\nu}(z), \, K_{\nu}(t)\sim \frac{2^{\nu-1}\Gamma(\nu)}{t^\nu}\text{ as } t\to 0,\text{ and }K_{\nu}(t)\sim \sqrt{\frac{\pi}{2t}}e^{-t}\text{ as } t\to \infty, \, \, \nu >0.\end{equation}
 Then \begin{align*}\int_0^\infty|\widehat{t\partial_t P^{\gamma}_{t}}(\theta)|^2 \frac{dt}{t}&\leq C_{\gamma}\int_0^{\infty} \biggl( t^{1+\gamma}\theta^{\frac{1+\gamma}{2}} K_{1-\gamma}(ct\sqrt{\theta}\cos(\varphi_0)) \biggr)^2 \frac{dt}{t}\\
&\leq C_{\gamma}\int_0^{\frac{1}{c\sqrt{\theta}\cos\varphi_0}} \frac{t^{2+2\gamma} \theta^{1+\gamma}}{t^{2-2\gamma}\theta^{1-\gamma}(\cos\varphi_0)^{2-2\gamma}}\frac{dt}{t}+ C_{\gamma}\int_{\frac{1}{c\sqrt{\theta}\cos\varphi_0}}^\infty\frac{t^{2+2\gamma} \theta^{1+\gamma}e^{-2ct\sqrt{\theta}\cos\varphi_0}}{t\sqrt{\theta}\cos\varphi_0}\frac{dt}{t}\\
&\underbrace{=}_{\substack{ct\sqrt{\theta}\cos\varphi_0}=u}C_\gamma\int_0^1 \frac{du}{u^{1-4\gamma}} +C_\gamma\int_1^{\infty} u^{2\gamma+1}e^{-2u}\frac{du}{u}<\infty.
\end{align*}
Observe that in the last identity we have used that $\sqrt{2}/2\le \cos\varphi_0\le 1$.

The case $\frac{7\pi}{4}<\theta< 2\pi$ follows analogously, by using that $|1-e^{i\theta}|=|1-e^{i(\theta-2\pi)}|\sim |\theta-2\pi|$, for $\theta$ being closed to $2\pi$.

On the other hand, if $\frac{\pi}{4}\leq \theta\leq \frac{7\pi}{4}$, then by using again \cite[2.3.16.1, p. 344]{Prudnikov} we get
\begin{align*}
|\widehat{t\partial_tP_t^\gamma}(\theta)|&\leq\frac{1}{\Gamma(\gamma)}\int_0^\infty e^{-r} \frac{t^2}{2r}|e^{i\theta}-1|e^{-\frac{t^2}{4r}\Re(1-e^{i\theta})}\frac{dr}{r^{1-\gamma}}\leq C_\gamma\int_0^\infty e^{-r} \frac{t^2}{2r}e^{-\frac{t^2}{4r}(1-\cos\theta)}\frac{dr}{r^{1-\gamma}}\\
&\leq C_\gamma\int_0^\infty e^{-r} \frac{t^2}{2r}e^{-\frac{t^2}{4r}(1-\sqrt{2}/2)}\frac{dr}{r^{1-\gamma}}\leq C_\gamma\frac{t^2}{2}t^{\gamma-1}K_{1-\gamma}\bigg(t\sqrt{1-\sqrt{2}/2}\bigg)\\
&=C_\gamma t^{1+\gamma}K_{1-\gamma}(\tilde{c} t).
\end{align*}
Hence by estimates \eqref{Mac}
\begin{align*}
\int_0^\infty|\widehat{t\partial_t P^{\gamma}_{t}}(\theta)|^2 \frac{dt}{t}&=C_\gamma\int_0^{1/\tilde{c}}\frac{t^{2+2\gamma}}{(\tilde{c}t)^{2-2\gamma}}\frac{dt}{t}+ C_\gamma\int_{1/\tilde{c}}^\infty\frac{t^{2+2\gamma} e^{-2\tilde{c}t}}{\tilde{c}t}\frac{dt}{t}\\
&\underbrace{=}_{\substack{\tilde{c}t=u}}C_\gamma\int_0^1 \frac{du}{u^{1-4\gamma}} +C_\gamma\int_1^{\infty} u^{2\gamma+1}e^{-2u}\frac{du}{u}<\infty.
\end{align*}

Therefore, we have proved that $\|\mathfrak{T}f\|^2_{\ell^2_{L^2((0,\infty), \frac{dt}{t})}}\leq C_{\gamma}\lVert f\rVert_2.$

By the representation of the Poisson kernel, (\ref{P.K.}), we can write, for all $n\in \Z$, $t\geq 0$,
$$
t\partial_tP_{t,\pm}^{\gamma}f(n)=\sum_{j=0}^\infty\frac{f(n\pm j)}{2\sqrt{\pi}j!}\int_0^\infty e^{-u}u^{j-\gamma}\left(2\gamma-\frac{t^2}{2u}\right)e^{-\frac{t^2}{4u}}t^{2\gamma}\frac{du}{u}=\sum_{j=0}^\infty t\partial_tP_t^\gamma(j)f(n\pm j).
$$

Then, by Minkowski's integral inequality, we get  
\begin{align*}
\|t\partial _t P_{t}^\gamma(j)\|_{L^2((0,\infty), \frac{dt}{t})}&=C_{\gamma}\left(\int_0^\infty\left|\int_0^\infty\left(2\gamma-\frac{t^2}{2u}\right)e^{-u}\frac{u^{j-\gamma}}{j!}e^{-\frac{t^2}{4u}}t^{2\gamma}\frac{du}{u}\right|^2\frac{dt}{t}\right)^{1/2}\\
&\leq C_{\gamma}\int_0^\infty\frac{u^j}{j!}e^{-u}\left( \int_0^\infty\left|2\gamma-\frac{t^2}{2u}\right|^2e^{-\frac{t^2}{2u}}\left(\frac{t^2}{u}\right)^{2\gamma}\frac{dt}{t} \right)^{1/2}\frac{du}{u}\\
&=C_{\gamma}\int_0^\infty\frac{u^j}{j!}e^{-u}I^{1/2}\frac{du}{u}, \, \, \; \text{for} \, j\geq 1.
\end{align*}
 Observe that

\begin{align*}
I&\leq C\int_{0}^{2\sqrt{\gamma u}}4\gamma^2e^{-\frac{t^2}{2u}}\left(\frac{t^2}{u}\right)^{2\gamma}\frac{dt}{t}+C\int_{2\sqrt{\gamma u}}^\infty \frac{t^4}{4u^2}e^{-\frac{t^2}{2u}}\left(\frac{t^2}{u}\right)^{2\gamma}\frac{dt}{t}\\
&\underbrace{\leq}_{\substack{\frac{t}{2\sqrt{\gamma u}}=v}}C_\gamma\int_0^1e^{-2\gamma v^2}v^{4\gamma}\frac{dv}{v}+C_\gamma\int_1^\infty v^{4+4\gamma} e^{-2\gamma v^2}\frac{dv}{v}\leq C_\gamma.
\end{align*}

Hence, if $j\geq 1$,
$$
\|t\partial _t P_{t}^\gamma(j)\|_{L^2((0,\infty),\frac{dt}{t})}\leq C_\gamma \int_0^\infty \frac{u^j}{j!}e^{-u}\frac{du}{u}=C_\gamma\frac{\Gamma(j)}{\Gamma(j+1)}\leq \frac{C_\gamma}{j},
$$
and $\|t\partial _t P_{t}^\gamma(j)\|_{L^2((0,\infty), \frac{dt}{t})}=0 \le \frac{C_\gamma}{|j|}$ for $j<0.$

Regarding (II.2), if $j\geq 1$,

\begin{equation*}
\begin{array}{l}
\displaystyle\|t\partial _t P_{t}^\gamma(j)-t\partial _t P_{t}^\gamma(j+1)\|_{L^2((0,\infty),\frac{dt}{t})}\\
\displaystyle=C_\gamma\left(\int_0^\infty\left|\int_0^\infty\left(2\gamma-\frac{t^2}{2u}\right)\frac{e^{-u}}{j!}u^{j}\left( 1-\frac{u}{j+1}\right)e^{-\frac{t^2}{4u}}\left(\frac{t^2}{u}\right)^{\gamma}\frac{du}{u}\right|^2\frac{dt}{t}\right)^{1/2}\\
\displaystyle=C_\gamma\left(\int_0^\infty\left|\int_0^\infty\frac{1}{j!}\left(2\gamma-\frac{t^2}{2u}\right)e^{-\frac{t^2}{4u}}\left(\frac{t^2}{u}\right)^{\gamma}\frac{1}{u} \frac{\partial_u(e^{-u}u^{j+1})}{j+1}du\right|^2\frac{dt}{t}\right)^{1/2}\\
\displaystyle=C_\gamma\left(\int_0^\infty\left|\int_0^\infty\frac{1}{(j+1)!}e^{-u}u^{j+1}\partial_u\left\{\left(2\gamma-\frac{t^2}{2u}\right)e^{-\frac{t^2}{4u}}\left(\frac{t^2}{u}\right)^{\gamma}\frac{1}{u} \right\}du\right|^2\frac{dt}{t}\right)^{1/2}.
\end{array}
\end{equation*}
Note that
\begin{align*}
\left|\partial_u\left\{\left(2\gamma-\frac{t^2}{2u}\right)e^{-\frac{t^2}{4u}}\left(\frac{t^2}{u}\right)^{\gamma}\frac{1}{u} \right\}\right|& \leq C_\gamma\frac{e^{-\frac{t^2}{4u}}}{u^2}\left(\frac{t^2}{u}\right)^{\gamma} \left[1+\left(\frac{t^2}{u}\right)+\left(\frac{t^2}{u}\right)^2\right].
\end{align*}
Hence, by Minkowski's inequality,
\begin{equation*}
\begin{array}{l}
\displaystyle\|t\partial _t P_{t}^\gamma(j)-t\partial _t P_{t}^\gamma(j+1)\|_{L^2((0,\infty),\frac{dt}{t})}\\
\displaystyle\leq \frac{C_\gamma}{(j+1)!}\int_0^\infty e^{-u}u^{j-1}\bigg(\int_0^\infty\bigg|e^{-\frac{t^2}{4u}}\left(\frac{t^2}{u}\right)^{\gamma} \left[1+\left(\frac{t^2}{u}\right)+\left(\frac{t^2}{u}\right)^2\right]\bigg|^2\frac{dt}{t} \bigg)^{1/2}du\\
\displaystyle\leq C_\gamma \frac{\Gamma(j)}{\Gamma(j+2)}\leq \frac{C_\gamma}{j^2}.
\end{array}
\end{equation*}

Finally since   $\|\mathfrak{T}f(n) \|_{L^2((0,\infty), \frac{dt}{t})} = g_{+}^{\gamma}f(n),$  the result follows for the maximal operator, by choosing $p_0=2,$  and $E= L^2((0,\infty), \frac{dt}{t}).$

$\hfill {\Box}$

\subsection*{\it Acknowledgments.} We thank Pablo Ra\'ul Stinga for fruitful conversations that have contributed to improve the contain of
the paper.

This work was done while the first author was visiting
the University Aut\'onoma of Madrid. He is grateful to the other authors for their kind hospitality.


\begin{thebibliography}{999}

\bibitem{ALMV} L. Abadias, C. Lizama, P. J. Miana and M. P. Velasco. {\it Ces\`{a}ro sums and algebra homorphisms of bounded operators.} To appear in Israel J. Math.



\bibitem{Caffarelli} M. Allen, L. Caffarelli and A. Vasseur. {\it A parabolic problem with a fractional time derivative.}  Arch. Ration. Mech. Anal. 221 (2016), no. 2, 603-630.


\bibitem{Caffarelli2} M. Allen, L. Caffarelli and A. Vasseur. {\it Porous medium flow with both a fractional potential pressure and fractional time derivative.} Arxiv: 1509.06325v1, (2015).




\bibitem{ABHN} W. Arendt, C. J. K. Batty, M. Hieber, F. Neubrander. { \it Vector-valued Laplace transforms and Cauchy problems.} Second edition, Monographs in Mathematics. {\bf 96}, Birkh\"auser (2011).



\bibitem{Bernardis} A. Bernardis, F. J. Mart\'in-Reyes, P. R. Stinga and J. L. Torrea. {\it Maximum principles, extension problem and inversion for non-local one-sided equations.}  J. Differential Equations, 260 no. 7, (2016), 6333-6362.

\bibitem{Caffarelli3} L. Caffarelli and Y. Sire {\it On some pointwise inequalities involving nonlocal operators.} Arxiv: 1604.05665v1, (2016).

\bibitem{CS} L. Caffarelli and L. Silvestre. {\it An extension problem related to the fractional Laplacian.} Comm. Partial Differential Equations. 32:1245-1260. (2007).


\bibitem{Chapman} S. Chapman, {\it On non-integral orders of summability of series and integrals.} Bull. Amer. Math. Soc. Volume 18, Number 3 (1911), 111-117.

\bibitem{Ciaurri} O. Ciaurri, L. Roncal, P. R. Stinga, J. L. Torrea and J. L. Varona. {\it Fractional discrete laplacian versus discretized fractional laplacian.} Arxiv: 1507.04986v1, (2015).

\bibitem{Castillo} D. del Castillo-Negrete, B.A. Carreras, and V.E. Lynch, {\it Fractional diffusion in plasma turbulence.}  Physics of Plasmas 11, 3854 (2004).

\bibitem{Gale} J. E. Gal\'e, P. J. Miana and P. R. Stinga. {\it Extension problem and fractional operators: semigroups and wave equations.} J. Evol. Equ. 13 (2013), no. 2, 343Ð368.



\bibitem{Kuttner} B.Kuttner, {\it On differences of fractional order}. Proceeding of the London Mathematical Society,
Vol.3, (1957), pp.453-466.

\bibitem{Lebedev} N. N. Lebedev. {\it Special functions and their applications.} Selected Russian Publications in the Mathematical Sciences. Prentice-Hall, INC., Englewood Cliffs, N.J., 1965.


\bibitem{MST} R. A. Mac\'ias, C. Segovia and J. L. Torrea,
{Singular integral operators with non-necessarily bounded kernels on spaces of homogeneous type},
\textit{Adv. Math.}
\textbf{93} (1992), 25--60.

\bibitem{Prudnikov} A. P. Prudnikov, Yu. A. Brychkov and O. I. Marichev. {\it Integrals and Series. Volumen 1, Elementary Functions.} USSR Academy of Sciences, Moscow, Taylor \& Francis, 1998.



\bibitem{RRT} J. L. Rubio de Francia, F. J. Ruiz and J. L. Torrea,
Calder\'on-Zygmund theory for operator-valued kernels,
\textit{Adv. in Math.}
\textbf{62} (1986), 7--48.



\bibitem{RT} F. J. Ruiz and J. L. Torrea,
{Vector-valued Calder\'on-Zygmund theory and Carleson measures on spaces of homogeneous nature},
\textit{Studia Math.}
\textbf{88} (1988), 221--243.


\bibitem{Samko} S. G. Samko, A. A. Kilbas and O. I. Marichev. { \it Fractional integrals and derivatives. Theory and applications.} Gordon and Breach Science Publications, Minsk, 1987.


\bibitem{Stein} E.M. Stein, {\it Singular Integrals and Differentiability Properties of Functions}.  Princeton Univ. Press, Princeton, NY, 1970.

\bibitem{ST} P. Stinga, J.L. Torrea { \it Extension problem and Harnack's inequality for some fractional operators.} Communications in Partial Differential Equations 35(10), (2010).
\bibitem{Yosida} K. Yosida. {\it Functional Analysis.} Fifth edition, A Series of Comprehensive Studies in Mathematics. {\bf 123}, Springer (1978).

\bibitem{Zygmund} A. Zygmund. {\it Trigonometric Series.} 2nd ed. Vols. I, II, Cambridge University Press, New York, 1959.

\end{thebibliography}
\end{document}